\documentclass[a4, 12pt]{amsart}
\usepackage{amssymb}
\usepackage{amstext}
\usepackage{amsmath}
\usepackage{amscd}
\usepackage{latexsym}
\usepackage{amsfonts}
\theoremstyle{plain}
\newtheorem{thm}{Theorem}[section]
\newtheorem*{thm*}{Theorem}
\newtheorem*{cor*}{Corollary}

\newtheorem{prop}[thm]{Proposition}
\newtheorem{lem}[thm]{Lemma}
\newtheorem{cor}[thm]{Corollary}
\newtheorem{claim}{Claim}
\newtheorem*{claim*}{Claim}

\theoremstyle{definition}
\newtheorem{defn}[thm]{Definition}
\newtheorem{ex}[thm]{Example}
\newtheorem{rem}[thm]{Remark}
\newtheorem{fact}[thm]{Fact}

\theoremstyle{remark}

\numberwithin{equation}{thm}
\def\Hom{\mathrm{Hom}}

\def\Ext{\mathrm{Ext}}

\def\Proj{\operatorname{Proj}}

\def\Im{\mathrm{Im}}

\def\a{\mathfrak a}

\def\e{\mathrm{e}}
\def\m{\mathfrak m}
\def\n{\mathfrak n}

\def\p{\mathfrak p}

\def\Z{\Bbb Z}

\def\H{\mathrm{H}}

\newcommand{\rme}{\mathrm{e}}

\newcommand{\rmE}{\mathrm{E}}

\newcommand{\rmH}{\mathrm{H}}

\newcommand{\rmK}{\mathrm{K}}

\newcommand{\rmR}{\mathrm{R}}

\newcommand{\rmT}{\mathrm{T}}
\newcommand{\rmU}{\mathrm{U}}

\newcommand{\calF}{\mathcal{F}}

\newcommand{\fkm}{\mathfrak{m}}
\newcommand{\fkn}{\mathfrak{n}}

\def\depth{\mathrm{depth}}
\def\Supp{\mathrm{Supp}}

\def\Ass{\mathrm{Ass}}
\def\Assh{\mathrm{Assh}}

\def\Spec{\mathrm{Spec}}

\def\hdeg{\operatorname{hdeg}}

\def\gr{\mbox{\rm gr}}

\begin{document}

\setlength{\baselineskip}{15pt}
%%%%%%%%%%%%%%%%%%%%%%%%%%%%%%%%%%%%%%%%%%%%%%%%%%%%%%%%%%%%%
\title{The first Euler characteristics versus  \\
the homological degrees}
\pagestyle{plain}
\author{Shiro Goto}
\address{Department of Mathematics, School of Science and Technology, Meiji University, 1-1-1 Higashi-mita, Tama-ku, Kawasaki 214-8571, Japan}
\email{goto@math.meiji.ac.jp}
\author{Kazuho Ozeki}
\address{Department of Mathematical Science, Faculty of Science, Yamaguchi University, 1677-1 Yoshida, Yamaguchi 753-8512, Japan}
\email{ozeki@Yamaguchi-u.ac.jp}

\thanks{{AMS 2010 {\em Mathematics Subject Classification:}
13H15, 13D40, 13H10.}}

\maketitle
%%%%%%%%%%%%%%%%%%%%%%%%%%%%%%%%%%%%%%%%%%%%%%%%%%%%%%%%%%%%%
%%%%%%%%%%%%%%%%%%%%%%%%%%%%%%%%%%%%%%%%%%%%%%%%%%%%%%%%%%%%%
\begin{abstract}Let $M$ be a finitely generated module over a Noetherian local ring.
This paper reports, for a given parameter ideal $Q$ for $M$, a criterion for the equality $\chi_1(Q;M)=\hdeg_Q(M)-\e_Q^0(M)$, where $\chi_1(Q;M)$, $\hdeg_Q(M)$, and $\e_Q^0(M)$ respectively denote the first Euler characteristic,  the homological degree, and the multiplicity of $M$ with respect to $Q$.
We also study homological torsions of $M$ and give a criterion for a certain equality of the first Hilbert coefficients of parameters and the homological torsions of $M$.

\end{abstract}

{\footnotesize
\tableofcontents
}
%%%%%%%%%%%%%%%%%%%%%%%%%%%%%%%%%%%%%%%%%%%%%%%%%%%%%%%%%%%%%
%%%%%%%%%%%%%%%%%%%%%%%%%%%%%%%%%%%%%%%%%%%%%%%%%%%%%%%%%%%%%
%%%%%%%%%%%%%%%%%%%%%%%%%%%%%%%%%%%%%%%%%%%%%%%%%%%%%%%%%%%%%
%%%%%%%%%%%%
%%%%%%%%%%%%%%%%%%%%%%%%%%%%%%%%%%%%%%%%%%%%%%%%%%%%%%%%%%%%%
%%%%%%%%%%%%%%%%%%%%%%%%%%%%%%%%%%%%%%%%%%%%%%%%%%%%%%%%%%%%%

\section{Introduction}

The notion of homological degree was introduced by W. V. Vasconcelos and his students \cite{DGV} in 1998, and since then,  many authors have been engaged in the development of the theory.
Recently, in \cite{GHV, GhGHOPV2, V3}, Ghezzi, Hong, Phuong, Vasconcelos, and the authors also made use of homological degrees to obtain bounds for the Hilbert coefficients of parameters. The purpose of our paper is to study the relationship between the first Euler characteristics and the homological degrees of modules. We also investigate the first Hilbert coefficients of parameters in connection with the homological torsions of modules.

To state the problems and the results of our paper, first of all, let us fix some of our terminology. Let $A$ be a Noetherian local ring with  maximal ideal $\fkm$ and $d = \dim A > 0$. Let $M$ be a finitely generated $A$-module with $s = \dim_AM$. For simplicity, throughout this paper, we assume that $A$ is $\m$--adically complete and the residue class field $A/\fkm$ of $A$  is infinite. 

For each $j \in \Z$ we set
$$
M_j  %[\rmH_\fkm^j(M)]^\vee 
= \Hom_A(\rmH_\fkm^j(M), E),
$$
where $E=\rmE_A(A/\fkm)$ denotes the injective envelope of $A/\fkm$ and $\H_{\m}^j(M)$ the $j$th local cohomology module of $M$ with respect to the maximal ideal $\m$.
Then $M_j$ is a finitely generated $A$-module with $\dim_A M_j \le j$ for all $j \in \Z$ (Fact \ref{fact1}).  Let $I$ be a fixed $\fkm$-primary ideal in $A$ and let $\ell_A(N)$ denote, for an $A$-module $N$, the length of $N$.
Then there exist integers $\{\e_I^i(M)\}_{0 \leq i \leq s}$ such that
$$\ell_A(M/I^{n+1}M)=\e_I^0(M)\binom{n+s}{s}-\e_I^1(M)\binom{n+s-1}{s-1}+\cdots+(-1)^s\e_I^s(M)$$
for all $n \gg 0$. We call $\e_I^i(M)$ the $i$-th Hilbert coefficient of $M$ with respect to $I$ and especially call the leading coefficient $\e_I^0(M)~(>0)$ the multiplicity of $M$ with respect to $I$.

The homological degree $\hdeg_I(M)$ of $M$ with respect to $I$ is inductively defined in the following way, according to the dimension $s = \dim_A M$ of $M$.

\begin{defn}\label{defn1}$($\cite{V2}$)$ For each finitely generated $A$-module $M$ with $s = \dim_A M$, we set 
$$
\hdeg_I(M) = \left\{
\begin{array}{lc}
\ell_A(M) & \mbox{if $s \le 0$},\\
\vspace{1mm}\\
\rme_I^0(M) + \sum_{j=0}^{s-1} \binom{s-1}{j}\hdeg_I(M_j) & \mbox{if $s > 0$}
\end{array}
\right.
$$
and call it the homological degree of $M$ with respect to $I$.
\end{defn}

When $M$ is a generalized Cohen-Macaulay $A$-module, that is a finitely generated $A$-module $M$ whose local cohomology module $\H_{\m}^j(M)$ is finitely generated for all $j\neq s = \dim_AM$, we have
$$\hdeg_I(M) = \e_I^0(M) + \Bbb I (M),$$ where $$\mathbb{I}(M) = \sum_{j=0}^{s-1}\binom{s-1}{j} \ell_A(\H_{\m}^j(M))$$ denotes the St\"{u}ckrad-Vogel invariant of $M$.

In this paper we need also the notion of homological torsion of modules.

\begin{defn}\label{defn2}
Let $M$ be a finitely generated $A$--module with $s= \dim_A M \ge 2$. We set $$\rmT_I^i(M) = \sum_{j=1}^{s-i}\binom{s-i-1}{j-1} \hdeg_I(M_j)$$
for each $1 \leq i \leq s-1$ and call them the homological torsions of $M$ with respect to $I$.
\end{defn}

Notice that the homological degrees $\hdeg_I(M)$ and torsions $\rmT_I^i(M)$ of $M$ with respect to $I$ depend only on the integral closure of $I$.

In this paper we study the first Euler characteristic of modules relative to parameters in connection with  homological degrees.
Let $Q = (a_1, a_2, \ldots, a_s)$ be a parameter ideal for $M$.
We denote by $\rmH_i(Q;M)$~($i \in \mathbb Z)$ the $i$--th homology module of
the Koszul complex $\rmK_{\bullet}(Q; M)$  generated
by the system $a_1, a_2, \ldots, a_s$ of parameters for $M$.
We set
\[\chi_1(Q;M) =  \sum_{i \ge 1}(-1)^{i-1}\ell_{A}(\rmH_i(Q;M))\] and call it
the first {\em Euler characteristic} of $M$ relative to $Q$; hence
\[\chi_1(Q;M) = \ell_A(M/QM) -\e_Q^0(M) \geq 0\] by a classical result of Serre (see \cite{AB}, \cite{Se}).

In \cite{GhGHOPV2}, it was proved that, for parameter ideals $Q$ for $M$, an upper bound
$$ \chi_1(Q;M) \leq \hdeg_Q(M)-\e_Q^0(M)$$
of $\chi_1(Q;M)$  (Proposition \ref{euler}). It seems natural to ask what happens on the parameters $Q$ for $M$, when the equality $$ \chi_1(Q;M) = \hdeg_Q(M)-\e_Q^0(M)$$  is attained. The first main result of this paper answers this question and is stated as follows, where the sequence $a_1,a_2,\ldots,a_d$ is said to be  a $d$-sequence on $M$, if the equality $$[(a_1,a_2,\ldots,a_{i-1})M:_M a_ia_j]=[(a_1,a_2,\ldots,a_{i-1})M:_M a_j]$$ holds true for all $1 \leq i \leq j \leq d$ (\cite{H}).

\begin{thm}\label{thm1}
Let $M$ be a finitely generated $A$-module with $d = \dim_AM$ and let $Q$ be a parameter ideal of $A$.
Then the following conditions are equivalent:
\begin{itemize}
\item[$(1)$] $\chi_1(Q;M)=\hdeg_Q(M)-\e_Q^0(M)$.
\item[$(2)$] The following two conditions are satisfied.
\begin{itemize}
\item[$(a)$]
$$
(-1)^i\e_Q^i(M) = \left\{
\begin{array}{cl}
\rmT_Q^i(M) & \mbox{if $1 \leq i \leq d-1$},\\
\vspace{0.5mm}\\
\ell_A(\H_{\m}^0(M)) & \mbox{if $i=d$}
\end{array}
\right.
$$
for all $1 \leq i \leq d$.
\item[$(b)$] $$\ell_A(M/Q^{n+1}M)=\sum_{i=0}^d(-1)^i\e_Q^i(M)\binom{n+d-i}{d-i}$$ for all $n \geq 0$.
\end{itemize}
\end{itemize}
When this is the case, we have the following$\mathrm{:}$
\begin{itemize}
\item[$(\mathrm{i})$] There exist elements $a_1,a_2,\ldots,a_d \in A$ such that $Q=(a_1,a_2,\ldots,a_d)$ and $a_1,a_2,\ldots,a_d$ forms a $d$-sequence on $M$,
\item[$(\mathrm{ii})$] $QM \cap \H_{\m}^0(M)=(0)$, and $Q\H_{\m}^i(M)=(0)$ for all $1 \leq i \leq d-2$.
\end{itemize}
\end{thm}

Theorem \ref{thm1} shows also that when $M$ is a generalized Cohen-Macaulay $A$-module with $d = \dim_AM$, the parameter ideal $Q$ of $M$ is standard, that is the equality $$\ell_A(M/QM)-\e_Q^0(M)=\mathbb{I}(M):=\sum_{j=0}^{d-1}\binom{d-1}{j}\ell_A(\H_{\m}^j(M))$$ holds true if and only if 
$$
(-1)^i\e_Q^i(M) = \left\{
\begin{array}{ll}
\sum_{j=1}^{d-i}\binom{d-i-1}{j-1}\ell_A(\H_{\m}^j(M)) & \mbox{if $1 \leq i \leq d-1$},\\
\vspace{0.5mm}\\
\ell_A(\H_{\m}^0(M)) & \mbox{if $i=d$}
\end{array}
\right.
$$
for all $1 \leq i \leq d$ and $\ell_A(M/Q^{n+1}M)=\sum_{i=0}^d(-1)^i\e_Q^i(M)\binom{n+d-i}{d-i}$ for all $n \geq 0$.

Our next purpose is to investigate the relationship between the first Hilbert coefficients and the homological torsions for modules. In \cite{GhGHOPV2}, it was proves that the lower bound
$$ \e_Q^1(M) \geq - \rmT_Q^1(M) $$
of the first Hilbert coefficient $\e_Q^1(M)$ in terms of the homological torsion $\rmT_Q^1(M)$ (Proposition \ref{e1}).
Here we notice that the inequality $0 \geq \e_Q^1(M)$ holds true for every parameter ideals $Q$ of $M$ (\cite[Theorem 3.5]{MSV}) and that $M$ is a Cohen-Macaulay $A$-module once $0 = \e_Q^1(M)$ for some parameter ideal $Q$, provided $M$ is unmixed (see \cite{GhGHOPV1}). Recall that $M$ is said to be unmixed, if $\dim A/\p=\dim_A M$ for all $\p \in \Ass_AM$ (since $A$ is assumed to be $\m$-adically complete). It seems now natural to ask what happens on the parameters $Q$ of $M$ which satisfy the equality $ \e_Q^1(M) = - \rmT_Q^1(M) $. The second main result of this paper answers the question and is stated as follows $($Theorem \ref{main2}$)$. 

\begin{thm}\label{thm2}
Let $M$ be a finitely generated $A$-module with $d = \dim_AM\geq 2$ and suppose that $M$ is unmixed. 
Let $Q$ be a parameter ideal of $A$. 
Then the following conditions are equivalent:
\begin{itemize}
\item[$(1)$] $\chi_1(Q;M)=\hdeg_Q(M)-\e_Q^0(M)$.
\item[$(2)$] $\e_Q^1(M)=-\rmT_Q^1(M)$.
\end{itemize}
When this is the case, we have the following$\mathrm{:}$
\begin{itemize}
\item[$(\mathrm{i})$] $(-1)^i\e_Q^i(M)=\rmT_Q^1(M)$ for $2 \leq i\leq d-1$ and $\e_Q^d(M)=0$,
\item[$(\mathrm{ii})$] $\ell_A(M/Q^{n+1}M)=\sum_{i=0}^d(-1)^i\e_{Q}^i(M)\binom{n+d-i}{d-i}$ for all $n \geq 0$,
\item[$(\mathrm{iii})$] there exist elements $a_1,a_2,\ldots,a_d \in A$ such that $Q=(a_1,a_2,\ldots,a_d)$ and $a_1,a_2,\ldots,a_d$ forms a $d$-sequence on $M$, and
\item[$(\mathrm{iv})$] $Q\H_{\m}^i(M)=(0)$ for all $1 \leq i \leq d-2$.
\end{itemize}
\end{thm}

Theorem \ref{thm2} also yields that when $M$ is a generalized Cohen-Macaulay $A$-module with $d = \dim_AM \ge 2$, the equality $$\e_Q^1(M) = -\sum_{j=1}^{d-1}\binom{d-2}{j-1}\ell_A(\H_{\m}^j(M))$$ holds true if and only if $Q$ is a standard parameter ideal for $M$, provided $\depth A >0$ (\cite[Korollar 3.2]{Sch}, \cite[Theorem 2.1]{GO1}).

We now briefly explain how this paper is organized.
In Section 2 we will summarize,  for the later use in this paper, some auxiliary results on the homological degrees and torsions.
We shall prove Theorem \ref{thm1} in Section 3 (Theorem \ref{main1}).
In Section 3 we will explore an example of  parameter ideals which satisfy the equality in Theorem \ref{thm1} (1).
Theorem \ref{thm2} will be proven in Section 4 (Theorem \ref{main2}).
Unless $M$ is unmixed, the implication $(2) \Rightarrow (1)$ in Theorem \ref{thm2} does not hold true in general. 
We will show in Section 4 an example of  parameter ideals $Q$ in a two-dimensional mixed local ring $A$ such that $\e_Q^1(A)=-\rmT_Q^1(A)$ but $\chi_1(Q;A) < \hdeg_Q(A)-\e_Q^0(A)$.

In what follows, unless otherwise specified, let $A$ be a Noetherian local ring with maximal ideal $\m$ and $d = \dim A>0$.
Let $M$ be a finitely generated $A$-module with $s=\dim_AM$.
We throughout assume that $A$ is $\m$--adically complete and the field $A/\m$ is infinite.
For each $\m$-primary ideal $I$ in $A$ we set
$$ R=\rmR(I)=A[It], \ \ R'=\rmR'(I)=A[It,t^{-1}], \ \ \mbox{and} \ \ \gr_I(A)=\rmR'(I)/t^{-1}\rmR'(I),$$
where $t$ is an indeterminate over $A$.
%Let $\mathcal{M}=\m R+R_+$ be the unique graded maximal ideal in $R=\rmR(I)$.

%%%%%%%%%%%%%%%%%%%%%%%%%%%%%%%%%%%%%%%%%%%%%%%%%%%%%%%%%%%%%%%%%%%%%%%%%%%%%%%%%%%%%%%%%%%%%%%%%%%%%%%%%%%%%%%%%%%%%%%%%%%%%%%%%%%%%%%%%%%%%%%%%%%%%%%%%%%%%%%%%%%%%%%%%%%%%%%%%%%%%%%%%%%%%%%%%%%%%%%%%%%%%%%%%%%%%%%%%%%%%%%%%%%%%%%%%%%%%%%%%%%%%%%%%%%%%%%%%%%%%%%%%%%%%%%%%%%%%%%%%%%%%%%%%%%%%%%%%%%%%%%%%%%%%%%%%%%%%%%%%%%%%%%%%%%%%%%%%%%%%%%%%%%%%%%%%%%%%%%%%%%%%%%%%%%%%%%%%%%%%%%%%%%%%%%%%%%%%%%%

\section{Preliminaries}

In this section we summarize some basic properties of homological degrees and torsions of modules, which we need throughout this paper.
Some of the results are known but let us  include brief proofs for the sake of completeness.

For each $j \in \Z$ we set
$$
M_j  %[\rmH_\fkm^j(M)]^\vee 
= \Hom_A(\rmH_\fkm^j(M), E),
$$
where $E=\rmE_A(A/\fkm)$ denotes the injective envelope of $A/\fkm$ and $\H_{\m}^j(M)$ the $j$th local cohomology module of $M$ with respect to $\m$.

We begin with the following.

\begin{fact}\label{fact1}
For each $j \in \Z$, $M_j$ is a finitely generated $A$-module with $\dim_A M_j \le j$, where $\dim_A (0) = -\infty$.
\end{fact}

\begin{proof}
As $A$ is complete, $A$ is a homomorphic image of a Gorenstein complete local ring $B$ with $\dim B=\dim A$, and  passing to $B$, without loss of generality we may assume that $A$ is a Gorenstein ring.
Let $\p\in \Supp_A M_j$. Then since
$$
M_j \cong \Ext_A^{d-j}(M,A)
$$
by  the local duality theorem,  we get
$$
\Ext_{A_\p}^{d-j}(M_\p , A_\p) \ne (0).
$$
Hence $$d -j \le \operatorname{injdim} A_\p = \dim A_\p.$$
Thus $\dim A/\p= d - \dim A_\p \le j$, whence $\dim_A M_j \le j$.
\end{proof}

We recall the definition of homological degrees.

\begin{defn}\label{hdeg}$($\cite{V2}$)$
For each finitely generated $A$-module $M$ with $s = \dim_A M$ and for each $\m$-primary ideal $I$ of $A$, we set 
$$
\hdeg_I(M) = \left\{
\begin{array}{lc}
\ell_A(M) & \mbox{if $s \le 0$},\\
\vspace{1mm}\\
\rme_I^0(M) + \sum_{j=0}^{s-1} \binom{s-1}{j}\hdeg_I(M_j) & \mbox{if $s > 0$}
\end{array}
\right.
$$
and call it the homological degree of $M$ with respect to $I$.
\end{defn}

Let us summarize some basic properties of $\hdeg_I(M)$.

\begin{fact}\label{fact2}
Let $M$ and $M'$ are finitely generated $A$-modules. Let $I$ be an $\m$-primary ideal in $A$.
Then $0 \le \hdeg_I(M) \in \Z$. We furthermore have the following:
\begin{itemize}
\item[(1)] $\hdeg_I(M) = 0$ if and only if $M=(0)$.
\item[(2)] If $M \cong M'$, then $\hdeg_I(M) = \hdeg_I(M')$.
\item[(3)] $\hdeg_I(M)$ depends only on the integral closure of $I$.
\item[(4)] If $M$ is a generalized Cohen-Macaulay $A$-module, then
$$\hdeg_I(M) - \rme_I^0(M) = \mathbb{I}(M)$$
and
$$ \ell_A(M/QM)-\e_Q^0(M) \leq \mathbb{I}(M)$$
for all parameter ideals $Q$ for $M$ $($\cite{STC}$)$, where $\mathbb{I}(M) = \sum_{j=0}^{s-1}\binom{s-1}{j} \ell_A(\H_{\m}^j(M))$ denotes the St\"{u}ckrad-Vogel invariant of $M$.
\end{itemize}
\end{fact}

The following result plays a key role in the analysis of homological degree.

\begin{lem}\label{xyz}$(${\cite[Proposition 3.18]{V2}}$)$
Let $0 \to X \to Y \to Z \to 0$ be an exact sequence of finitely generated $A$-modules. Then the following assertions hold true:
\begin{itemize}
\item[(1)] If $\ell_A(Z) < \infty$, then $\hdeg_I(Y) \le \hdeg_I(X) + \hdeg_I(Z)$.
\item[(2)] If $\ell_A(X) < \infty$, then $\hdeg_I(Y) = \hdeg_I(X) + \hdeg_I(Z)$.
\end{itemize}
\end{lem}

\begin{proof}
(1) We may assume $\ell_A(Y)=\infty$; hence $\ell_A(X)=\infty$. Let $s=\dim_AX=\dim_AY >0$.
Then because $\ell_A(Z) < \infty$, we have the exact sequence
$$ 0 \to \H_{\m}^0(X) \to \H_{\m}^0(Y) \to Z \to \H_{\m}^1(X) \to \H_{\m}^1(Y) \to 0 $$
and isomorphisms $\H_{\m}^j(X) \cong \H_{\m}^j(Y)$~$(j \geq 2)$. Therefore, taking the Matlis dual, we get the exact sequence
$$ 0 \to Y_1 \to X_1 \to Z \to Y_0 \to X_0 \to 0 \ \ \ (\sharp_1)$$
and isomorphisms
$Y_j \cong X_j $~$(j \geq 2)$. Since $\e_I^0(X)=\e_I^0(Y)$ and $\hdeg_I(X_j)=\hdeg_I(Y_j)$ for $j \geq 2$, we have
\begin{eqnarray*}
&&\{\hdeg_I(X) + \hdeg_I(Z)\}-\hdeg_I(Y)\\
&=&\left\{ \rme_I^0(X)+\sum_{j=0}^{s-1}\binom{s-1}{j}\hdeg_I(X_j) + \ell_A(Z) \right\}- \rme_I^0(Y)-\sum_{j=0}^{s-1}\binom{s-1}{j}\hdeg_I(Y_j)\\
&=& (s-1)\{ \hdeg_I(X_1) -\hdeg_I(Y_1) \}+\hdeg_I(X_0)+\ell_A(Z)-\hdeg_I(Y_0)\\
&=& (s-1)\{ \hdeg_I(X_1) -\hdeg_I(Y_1) \}+\ell_A(X_0)+\ell_A(Z) -\ell_A(Y_0).
\end{eqnarray*}
Because $\ell_A(Y_0) \leq \ell_A(X_0)+\ell_A(Z)$ by  exact sequence $(\sharp_1)$, it is enough to show that $\hdeg_I(Y_1) \leq \hdeg_I(X_1)$ in the case where $s \geq 2$. If $\ell_A(X_1) < \infty$, this is clear; see exact sequence $(\sharp_1)$.

Assume that $\dim_AX_1=\dim_AY_1=1$.
We then have 
$$\hdeg_I(X_1)=\e_I^0(X_1) +\hdeg_I([X_1]_0)=\e_I^0(X_1) +\ell_A(\H_{\m}^0(X_1)) ~\text{and}~$$
$$\hdeg_I(Y_1)=\e_I^0(Y_1) +\hdeg_I([Y_1]_0)=\e_I^0(Y_1) +\ell_A(\H_{\m}^0(Y_1)).$$
Because $\e_I^0(X_1)=\e_I^0(Y_1)$ and $\ell_A(\H_{\m}^0(Y_1)) \leq \ell_A(\H_{\m}^0(X_1))$ by exact sequence $(\sharp_1)$, the inequality $\hdeg_I(Y_1) \leq \hdeg_I(X_1)$ follows, which proves assertion (1).

(2) We may assume $\ell_A(Y)=\infty$.
Hence  $\ell_A(Z)=\infty$. Let $s=\dim_AY=\dim_AZ >0$.
This time, because $\ell_A(X) < \infty$, we have the exact sequence
$$ 0 \to X \to \H_{\m}^0(Y) \to \H_{\m}^0(Z) \to 0$$ and isomorphisms $\H_{\m}^j(Y) \cong \H_{\m}^j(Z)$~$(j \geq 1)$. We take the Matlis dual to yield the exact sequence
$$ 0 \to Z_0 \to Y_0 \to X \to 0 \ \ \ (\sharp_2)$$
and isomorphisms
$Z_j \cong Y_j $~$(j \geq 1)$. Then because $\e_I^0(Y)=\e_I^0(Z)$, $\hdeg_I(Y_j)=\hdeg_I(Z_j)$ for all $j \geq 1$, and $\ell_A(X)+\ell_A(Z_0)=\ell_A(Y_0)$ by exact sequence $(\sharp_2)$, we get
\begin{eqnarray*}
&&\{\hdeg_I(X) + \hdeg_I(Z)\}-\hdeg_I(Y) \\
&=& \left\{\ell_A(X) + \rme_I^0(Z)+\sum_{j=0}^{s-1}\binom{s-1}{j}\hdeg_I(Z_j) \right\}-\rme_I^0(Y)-\sum_{j=0}^{s-1}\binom{s-1}{j}\hdeg_I(Y_j)\\
&=& \ell_A(X)+\hdeg_I(Z_0)-\hdeg_I(Y_0)\\
&=& \ell_A(X)+\ell_A(Z_0)-\ell_A(Y_0)=0,
\end{eqnarray*}
which shows assertion (2).
\end{proof}

\begin{rem}\label{remark}
In Lemma \ref{xyz} (1) the equality $$\hdeg_I(Y) = \hdeg_I(X) + \hdeg_I(Z)$$ does not hold true in general, even though $\ell_A(Z) < \infty$. To see this, suppose that  $A$ is a Cohen-Macaulay local ring with $\dim A =1$ and consider the exact sequence $0 \to \fkm \to A \to A/\fkm\to 0.$ Then because $\m$ is a Cohen-Macaulay $A$--module, we have
$\hdeg_I(A) = \rme_I^0(A) = \rme_I^0(\fkm) = \hdeg_I(\fkm),$ so that  $$\hdeg_I(A) < \hdeg_I(A) +1 = \hdeg_I(\m) + \hdeg_I(A/\m)$$
 as $\hdeg_I(A/\fkm) = 1$.
\end{rem}

Let $R=\rmR(I)=A[It] \subseteq A[t]$ be the Rees algebra of $I$ (here $t$ denotes an indeterminate over $A$) and let 
$
f : I \to R, ~a \mapsto at
$
be the identification of $I$ with $R_1 = It$. Set $$\Proj R = \{\p \mid \p \text{ is a graded prime ideal of} ~R ~\text{such that}~\p \not\supseteq R_+\}.$$
We then  have the following.

\begin{lem}\label{superficial1}$(${\cite[Theorem 2.13]{V1}}$)$
Let $M$ be a finitely generated $A$-module. Then there exists a finite subset $\calF \subseteq \Proj R$ such that 
\begin{enumerate}
\item[$(1)$] every $a \in I \setminus \bigcup_{\p \in \calF}[f^{-1}(\p) + \fkm I]$ is superficial for $M$ with respect to $I$ and 
\item[$(2)$] $\hdeg_I(M/aM) \le \hdeg_I(M)$ for each $a \in I \setminus \bigcup_{\p \in \calF}[f^{-1}(\p) + \fkm I]$.
\end{enumerate}
\end{lem}

\begin{proof}
We proceed by induction on $s=\dim_A M$. 
If $s \le 0$, choose $\calF = \emptyset$.
Suppose $s=1$ and let $\calF = \{\p \in \Ass_R \gr_I(M) \mid \p \not\supseteq R_+\}$. Then every $a \in I \setminus \bigcup_{\p \in \calF}[f^{-1}(\p) + \fkm I]$ is superficial for $M$ with respect to $I$. Set $W=\rmH_\fkm^0(M)$ and $M'=M/W$. Then, because $M'$ is a Cohen-Macaulay $A$-module, we obtain the exact sequence
$$
0 \to W/aW \to M/aM \to M'/aM' \to 0
$$
from the canonical exact sequence 
$$0 \to W \to M \to M' \to 0.$$
Hence 
\begin{eqnarray*}
\hdeg_I(M/aM) &=& \ell_A(M/aM) \\
&=& \ell_A(M'/aM')+\ell_A(W/aW)\\
&\le& \rme_{(a)}^0(M') + \ell_A(W) \\
&=& \rme_I^0(M) + \hdeg_I(M_0) \\
&=& \hdeg_I(M).
\end{eqnarray*}

Suppose now that $s \geq 2$ and that our assertion holds true for $s-1$.
Because $\dim_A M_j \leq j$ for all $0 \leq j \leq s-1$ (Fact \ref{fact1}), the hypothesis of induction on $s$ shows that there exists a finite subset $\calF \subseteq \Proj R$ such that every $a \in I \setminus \bigcup_{\p \in \calF}[f^{-1}(\p) + \fkm I]$ is superficial for all $M$ and $M_j$ with respect to $I$ and $\hdeg_I(M_j/aM_j) \le \hdeg_I(M_j)$ for all $ 1 \leq j \leq s-1$. We take an element $a \in I \setminus \bigcup_{\p \in \calF}[f^{-1}(\p) + \fkm I]$ and set $\overline{M} = M/aM$. Consider the long exact sequence
$$
0 \to (0):_Ma \to \rmH_\fkm^0(M) \overset{a}{\to} \rmH_\fkm^0(M) \to \rmH_\fkm^0(\overline{M}) \to \rmH_\fkm^1(M) \overset{a}{\to} \rmH_\fkm^1(M) \to \rmH_\fkm^1(\overline{M}) \to \cdots
$$
$$ \cdots \to \rmH_\fkm^j(M) \overset{a}{\to} \rmH_\fkm^j(M) \to \rmH_\fkm^j(\overline{M}) \to \rmH_\fkm^{j+1}(M) \overset{a}{\to} \rmH_\fkm^{j+1}(M) \to \cdots $$
of local cohomology modules induced from the exact sequence
$$ 0 \to (0):_M a \to M \overset{a}{\to} M \to \overline{M} \to 0. $$
Then, taking the Matlis dual of the above long exact sequence, we get exact sequences
$$ 0 \to M_{j+1}/aM_{j+1} \to \overline{M}_j \to (0):_{M_j}a \to 0 $$
and embeddings 
$$ 0 \to (0):_{M_j}a \to M_j $$
%$$
%\matrix{
%& & & & 0 \ar[d] &\\
%(\#) & 0 \ar[r] & M_{i+1}/aM_{i+1} \ar[r] & \overline{M}_i \ar[r] & (0):_{M_i}a \ar[r] \ar[d] & 0\\
% & & & & M_i &
%}
%$$
for all $0 \le j \le s-2$. 
Consequently, because $\ell_A((0):_{M_j}a) < \infty$,  by Lemma \ref{xyz} we have
\begin{eqnarray*}
\hdeg_I(\overline{M}_j) &\le& \hdeg_I([(0):_{M_j}a]) + \hdeg_I(M_{j+1}/aM_{j+1})\\
&\le& \hdeg_I(M_j) + \hdeg_I(M_{j+1})
\end{eqnarray*}
for each $0 \leq j \leq s-2$, so that 
\begin{eqnarray*}
\hdeg_I(\overline{M}) &=& \rme_I^0(\overline{M}) + \sum_{j=0}^{s-2} \binom{s-2}{j} \hdeg_I(\overline{M}_j)\\
&\le& \rme_I^0(M) + \sum_{j=0}^{s-2}\binom{s-2}{j} \left\{\hdeg_I(M_j) +\hdeg_I(M_{j+1})\right\}\\
&=& \rme_I^0(M) + \sum_{j=0}^{s-1} \binom{s-1}{j} \hdeg_I(M_j)\\
&=& \hdeg_I(M),
\end{eqnarray*}
because $\e_I^0(\overline{M})=\e_I^0(M)$. Hence the result follows.
\end{proof}

\begin{defn}\label{torsion}
Let $M$ be a finitely generated $A$--module with $s = \dim_A M \ge 2$. We set $$\rmT_I^i(M) = \sum_{j=1}^{s-i}\binom{s-i-1}{j-1} \hdeg_I(M_j)$$
for each $1 \leq i \leq s-1$ and call them the homological torsions of $M$ with respect to $I$.
\end{defn}

We notice that $\rmT_I^j(M)$ depends only on the integral closure of $I$.

\begin{lem}\label{superficial2}
Let $M$ be a finitely generated $A$--module with $s = \dim_A M \ge 3$ and $I$  an $\m$-primary ideal of $A$. 
Then, for each $1 \leq i \leq s-2$, there exists a finite subset $\calF \subseteq \Proj R$ such that every $a \in I \setminus \bigcup_{\p \in \calF}[f^{-1}(\p) + \fkm I]$ is superficial for $M$ with respect to $I$, satisfying the inequality 
$$\rmT_I^i(M/aM) \le \rmT_I^i(M).$$
\end{lem}

\begin{proof}
Let $1 \leq i \leq s-2$.
Thanks to Lemma \ref{superficial1}, there exists a subset $\calF \subseteq \Proj R$ such that every $a \in I \setminus \bigcup_{\p \in \calF}[f^{-1}(\p)+\m I]$ is superficial for  $M$ and $M_j$ with respect to $I$ and $\hdeg_I(M_j/aM_j) \leq \hdeg_I(M_j)$ for all $1 \leq j \leq s-i-1$. Set $\overline{M}=M/aM$.
Then, by the same argument as is in the proof of Lemma \ref{superficial1}, for all $1 \leq j \leq s-i-1$ we get  inequalities
$$\hdeg_I(\overline{M}_j) \le \hdeg_I(M_j) + \hdeg_I(M_{j+1}).$$ Hence 
\begin{eqnarray*}
\rmT_I^i(\overline{M}) &=& \sum_{j=1}^{s-i-1} \binom{s-i-2}{j-1} \hdeg_I(\overline{M}_j)\\
&\le& \sum_{j=1}^{s-i-1}\binom{s-i-2}{j-1} \left\{\hdeg_I(M_j) +\hdeg_I(M_{j+1})\right\}\\
&=& \sum_{j=1}^{s-i} \binom{s-i-1}{j-1} \hdeg_I(M_j)\\
&=& \rmT_I^i(M),
\end{eqnarray*}
as claimed.
\end{proof}

\section{Relation between  the first Euler characteristics and the homological degrees}

In this section we study the relation between the first Euler characteristics and the homological degrees. Let $Q = (a_1, a_2, \ldots, a_s)$ be a parameter ideal for $M$.
We denote by $\rmH_i(Q;M)$~($i \in \mathbb Z)$ the $i$--th homology  module of
the Koszul complex $\rmK_{\bullet}(Q; M)$  generated
by the system $a_1, a_2, \ldots, a_s$ of parameters for $M$. Set
\[\chi_1(Q;M) =  \sum_{i \ge 1}(-1)^{i-1}\ell_{A}(\rmH_i(Q;M))\] and call it
the first Euler characteristic of $M$ relative to $Q$. Hence
\[\chi_1(Q;M) = \ell_A(M/QM) -\e_Q^0(M) \geq 0.\]

We note the following.

\begin{lem}\label{/a}
Let $M$ be a finitely generated $A$-module with $s=\dim_AM \geq 2$.
Let $Q$ be a parameter ideal for $M$ and assume that $a \in Q \backslash \m Q$ is a superficial element for $M$ with respect to $Q$.
Then $\chi_1(Q;M)=\chi_1(\overline{Q};\overline{M})$, where $\overline{M}=M/aM$ and $\overline{Q}=Q/(a)$.
\end{lem}

\begin{proof}
We get $\chi_1(Q;M)=\ell_A(M/QM)-\e_Q^0(M)=\ell_A(\overline{M}/Q\overline{M})-\e_Q^0(\overline{M})=\chi_1(\overline{Q};\overline{M})$, as $\e_Q^0(M)=\e_Q^0(\overline{M})$ and $\ell_A(M/QM)=\ell_A(\overline{M}/Q\overline{M})$.
\end{proof}

The following inequality is due to \cite{GhGHOPV2}.
We indicate a brief proof for the sake of completeness.

\begin{prop}\label{euler}$($\cite[Theorem 7.2]{GhGHOPV2}$)$
Let $M$ be a finitely generated $A$-module with $d=\dim_AM$.
Then $$\chi_1(Q;M) \leq \hdeg_Q(M)-\rme_Q^0(M)$$
for every parameter ideal $Q$ of $A$.
\end{prop}

\begin{proof}
Suppose $d=1$. Then $M$ is a generalized Cohen-Macaulay $A$-module and hence
$$\chi_1(Q;M)=\ell_A(M/QM)-\e_Q^0(M) \leq \ell_A(\H_{\m}^0(M))=\hdeg_Q(M_0)=\hdeg_Q(M)-\rme_Q^0(M)$$
by Fact \ref{fact2} (4). Assume that $d \geq 2$ and that our assertion holds true for $d-1$.
We choose an element $a \in Q \backslash \m Q$ so that $a$ is  superficial for $M$ with respect to $Q$ and $\hdeg_Q(M/aM) \leq \hdeg_Q(M)$ (Lemma \ref{superficial1}). Then, setting $\overline{M}=M/aM$, by Lemma \ref{/a} and the hypothesis of induction on $d$ we get 
$$ \chi_1(Q;M)=\chi_1(\overline{Q};\overline{M}) \leq \hdeg_{Q}(\overline{M})-\e_Q^0(\overline{M}) \leq \hdeg_Q(M)-\e_Q^0(M).$$
\end{proof}

It seems  natural to ask what happens on the parameters $Q$ for $M$, when $\chi_1(Q;M)=\hdeg_Q(M)-\e_Q^0(M)$. The following   theorem answers the question, which is the main result of this section.

\begin{thm}\label{main1}
Let $M$ be a finitely generated $A$-module with $d=\dim_AM$.
Let $Q$ be a parameter ideal of $A$.
Then the following conditions are equivalent:

\begin{itemize}
\item[$(1)$] $\chi_1(Q;M)=\hdeg_Q(M)-\e_Q^0(M)$.
\item[$(2)$] The following conditions are satisfied:
\begin{itemize}
\item[$(\mathrm{a})$]
$$
(-1)^i\e_Q^i(M) = \left\{
\begin{array}{cl}
\rmT_Q^i(M) & \mbox{if $1 \leq i \leq d-1$},\\
\vspace{0.5mm}\\
\ell_A(\H_{\m}^0(M)) & \mbox{if $i=d$}
\end{array}
\right.
$$
for all $1 \leq i \leq d$ and
\item[$(\mathrm{b})$] $$\ell_A(M/Q^{n+1}M)=\sum_{i=0}^d(-1)^i\e_Q^i(M)\binom{n+d-i}{d-i}$$ for all $n \geq 0$.
\end{itemize}
\end{itemize}
When this is the case, we have the following$\mathrm{:}$
\begin{itemize}
\item[$(\mathrm{i})$] There exist elements $a_1,a_2,\ldots,a_d \in A$ such that $Q=(a_1,a_2,\ldots,a_d)$ and $a_1,a_2,\ldots,a_d$ forms a $d$-sequence on $M$.
\item[$(\mathrm{ii})$] $QM \cap \H_{\m}^0(M)=(0)$ and $Q\H_{\m}^i(M)=(0)$ for all $1 \leq i \leq d-2$.
\end{itemize}
\end{thm}

To prove Theorem \ref{main1}, we need the following:

\begin{lem}\label{W}
Let $M$ be a finitely generated $A$-module with $d = \dim_AM$ and let $Q$ be a parameter ideal of $A$.
Then $\chi_1(Q;M)=\hdeg_Q(M)-\e_Q^0(M)$ if and only if $\chi_1(Q;M/\H_{\m}^0(M))=\hdeg_Q(M/\H_{\m}^0(M))-\e_Q^0(M/\H_{\m}^0(M))$ and $QM \cap \H_{\m}^0(M)=(0)$.
\end{lem}

\begin{proof}
We set $W=\H_{\m}^0(M)$ and $M'=M/W$. Consider the exact sequence
$$ 0 \to W/[QM \cap W] \to M/QM \to M'/QM' \to 0 \ \ \ \ \ (\sharp_3) $$
obtained by  the canonical exact sequence
$$ 0 \to W \to M \to M' \to 0. $$ Assume that $\chi_1(Q;M)=\hdeg_Q(M)-\e_Q^0(M)$. We then have
\begin{eqnarray*}
\chi_1(Q;M)&=&\ell_A(M/QM) -\e_Q^0(M)\\
&=& \{\ell_A(M'/QM')+\ell_A(W/[QM \cap W])\} -\e_Q^0(M')\\
&=& \chi_1(Q;M')+\{\ell_A(W)-\ell_A(QM \cap W)\}\\
%&\leq& \hdeg_Q(M') - \e_Q^0(M')+\{\ell_A(W)-\ell_A(QM \cap W)\}\\
&\leq& \{\hdeg_Q(M')-\e_Q^0(M')\}+ \ell_A(W)\\
&=& \hdeg_Q(M)-\e_Q^0(M) = \chi_1(Q;M),
\end{eqnarray*}
because $\e_Q^0(M)=\e_Q^0(M')$, $\ell_A(M/QM)=\ell_A(M'/QM')+\ell_A(W/QM \cap W)$ by exact sequence $(\sharp_3)$, $\chi_1(Q;M') \leq \hdeg_Q(M') - \e_Q^0(M')$ by Proposition \ref{euler}, and $\hdeg_Q(M)=\hdeg_Q(M')+ \ell_A(W)$ by Lemma \ref{xyz} (2). Thus $\chi_1(Q;M')=\hdeg_Q(M')-\e_Q^0(M')$ and $QM \cap W=(0)$. Conversely, assume that $\chi_1(Q;M')=\hdeg_Q(M')-\e_Q^0(M')$ and $QM \cap W=(0)$.
Then 
\begin{eqnarray*}
\chi_1(Q;M)&=&\ell_A(M/QM) -\e_Q^0(M)\\
&=& \{\ell_A(M'/QM')+\ell_A(W)\} -\e_Q^0(M')\\
&=& \chi_1(Q;M')+\ell_A(W)\\
&=& \{\hdeg_Q(M')-\e_Q^0(M')\}+ \ell_A(W)\\
&=& \hdeg_Q(M)-\e_Q^0(M),
\end{eqnarray*}
because $\e_Q^0(M)=\e_Q^0(M')$, $\ell_A(M/QM)=\ell_A(M'/QM')+\ell_A(W)$ by exact sequence $(\sharp_3)$, and $\hdeg_Q(M)=\hdeg_Q(M')+ \ell_A(W)$ by Lemma \ref{xyz} (2), which  proves Lemma \ref{W}.
\end{proof}

The following result shows that Theorem \ref{main1} (i) holds true, once $\chi_1(Q;M)=\hdeg_Q(M)-\e_Q^0(M)$.

\begin{prop}\label{d-seq}
Let $M$ be a finitely generated $A$-module with $d = \dim_AM$ and $Q$  a parameter ideal of $A$.
Let $a_1 \in Q \backslash \m Q$ be a superficial element for $M$ with respect to $Q$ such that $\hdeg_Q(M/a_1M) \leq \hdeg_Q(M)$.
Assume that $$\chi_1(Q;M)=\hdeg_Q(M)-\e_Q^0(M).$$
Then there exist elements $a_2,a_3,\ldots,a_d \in A$ such that $Q=(a_1, a_2,\ldots,a_d)$ and $a_1, a_2,\ldots,a_d$ forms a $d$-sequence on $M$.
\end{prop}

\begin{proof}
We proceed by induction on $d$.
Suppose that $d=1$. Then since $M$ is generalized Cohen-Macaulay and the ideal $Q$ is standard $($\cite[Theorem 2.1]{T}$)$, $a_1$ certainly forms a $d$-sequence on $M$ (\cite[Proposition 2.7]{T}). Assume that $d \geq 2$ and that our assertion holds true for $d-1$. Set $\overline{M}=M/a_1M$, $\overline{A}=A/(a_1)$, and $\overline{Q}=Q/(a_1)$.
Then since
$$ \chi_1(Q;M)= \chi_1(\overline{Q};\overline{M}) \leq \hdeg_Q(\overline{M}) -\e_Q^0(\overline{M}) \leq \hdeg_Q(M)-\e_Q^0(M)=\chi_1(Q;M) $$
by Lemma \ref{/a} and Proposition \ref{euler}, we have $ \chi_1(\overline{Q};\overline{M})=\hdeg_Q(\overline{M}) -\e_Q^0(\overline{M}) $. Because the residue class field $A/\m$ of $A$ is infinite, we may also choose an element $a_2 \in Q$ so that $a_2$ is  superficial for $\overline{M}$ with respect to $\overline{Q}$, $\hdeg_Q(\overline{M}/a_2\overline{M}) \leq \hdeg_Q(\overline{M})$ ~(Lemma \ref{superficial1}), and $a_1, a_2$ forms, furthermore, a part of a minimal system of generators of $Q$.
Then the hypothesis of induction on $d$ guarantees  that there exist elements $a_3,a_4,\ldots,a_d \in A$ such that $\overline{Q}=(a_2,a_3,\ldots,a_d)\overline{A}$ and $a_2,a_3,\ldots,a_d$ forms a $d$-sequence on $\overline{M}$.
Hence 
$$ [(a_2,a_3,\ldots,a_{i-1})\overline{M}:_{\overline{M}}a_ia_j]= [(a_2,a_3,\ldots,a_{i-1})\overline{M}:_{\overline{M}}a_j],$$
so that
$$ [(a_1,a_2,\ldots,a_{i-1}){M}:_{M}a_ia_j]= [(a_1,a_2,\ldots,a_{i-1}){M}:_{M}a_j]$$
for all $2 \leq i \leq j \leq d$. It is now enough to show that $[(0):_M a_1a_j]=[(0):_M a_j]$ for all $1 \leq j \leq d$.
Take $m \in [(0):_M a_1a_j]$. Then $a_1a_jm=0$.
Then by Lemma \ref{W}  $$a_jm \in [(0):_M a_1] \cap QM \subseteq W \cap QM=(0),$$ because $a_1$ is superficial for $M$.
Hence $m \in [(0):_M a_j]$, so that $[(0):_M a_1a_j] \subseteq [(0):_M a_j]$.
Thus $[(0):_M a_1a_j]=[(0):_M a_j]$ for all $1 \leq j \leq d$. Hence $a_1,a_2,\ldots,a_d$ forms a $d$-sequence on $M$.
\end{proof}

The following result is due to \cite{GNi}. See, for example, \cite[Proposition 3.1]{MSV} for a proof.

\begin{lem}\label{d=1}$(${\cite[Lemma 2.4]{GNi}}$)$
Let $M$ be a finitely generated $A$-module with $\dim_AM=1$. Then $\e_Q^1(M)=-\ell_A(\H_{\m}^0(M))$ for every parameter ideal $Q$ of $M$.
\end{lem}

The following result is, more or less, known. Let us indicate a brief proof for the sake of completeness, because it plays a key role in our proof of Theorem \ref{main1}.

\begin{prop}[{cf. \cite[Proposition 3.4]{GO2}}]\label{d-poly}
Let $M$ be a finitely generated $A$-module with $d = \dim_AM$. Let  $Q=(a_1,a_2,\ldots, a_d)$ be a parameter ideal of $A$ and assume that  $a_1,a_2,\ldots,a_d$ forms a $d$-sequence on $M$.
Then we have the following, where $Q_i =(a_1,a_2,\ldots,a_i)$ for each $0 \leq i \leq d$.
\begin{itemize}
\item[$(1)$] $\e_Q^0(M)=\ell_A(M/QM)-\ell_A \left([Q_{d-1}M:_Ma_d]/Q_{d-1}M\right)$.
\item[$(2)$] $(-1)^i\e_Q^i(M) = \ell_A(\H_{\m}^0(M/Q_{d-i}M))-\ell_A(\H_{\m}^0(M/Q_{d-i-1}M))$ for $1 \leq i \leq d-1$ and $(-1)^{d}\e_Q^{d}(M) = \ell_A(\H_{\m}^0(M))$.
\item[$(3)$] $\ell_A(M/Q^{n+1}M)=\sum_{i=0}^d(-1)^i\e_Q^i(M)\binom{n+d-i}{d-i}$ for all $n \geq 0$.
\end{itemize}
\end{prop}

\begin{proof}
Since  $A$ is complete, there exists a surjective homomorphism  $\varphi:B \to A$ of rings, where $B$ is a Cohen-Macaulay complete local ring with $\dim B=\dim A$ and a system $\alpha_1,\alpha_2,\ldots,\alpha_d$ of parameters of $B$ such that $\varphi(\alpha_i)=a_i$ for all $1 \leq i \leq d$.
Therefore, passing to the ring $B$, we may assume that $A$ is a Cohen-Macaulay ring. Let $C=A \ltimes M$ denote the idealization of $M$ over $A$.
Then $C$ is a Noetherian local ring with maximal ideal $\n=\m \ltimes M$ and $\dim C= d$. We have
$$\ell_A(C/Q^{n+1}C) = \ell_A(A/Q^{n+1}) + \ell_A(M/Q^{n+1}M)$$
for all $n \geq 0$ and  $a_1,a_2,\ldots,a_d$ forms a $d$-sequence on $C$, because $A$ is a Cohen-Macaulay ring.
Hence
\begin{eqnarray*}
\ell_A(M/Q^{n+1}M) &=&  \ell_C(C/Q^{n+1}C) - \ell_A(A/Q^{n+1})\\
&=& \sum_{i=0}^d(-1)^i\e_{QC}^i(C)\binom{n+d-i}{d-i} - \ell_A(A/Q)\binom{n+d}{d}
\end{eqnarray*}
for all $n \geq 0$ by \cite[Proposition 3.4]{GO2}.
Therefore $\e_{Q}^0(M)=\e_{QC}^0(C)-\ell_A(A/Q)$ and $\e_{Q}^i(M)=\e_{QC}^i(C)$ for all $1 \leq i \leq d$. Since $a_1,a_2,\ldots,a_d$ is a $d$-sequence on $C$ and $A$ is a Cohen-Macaulay ring, by \cite[Proposition 3.4]{GO2} we get 
\begin{eqnarray*}
\e_{QC}^0(C)&=&\ell_C(C/QC)-\ell_C \left([Q_{d-1}C:_Ca_d]/Q_{d-1}C\right)\\
&=&\{\ell_A(A/Q)+\ell_A(M/QM)\}-\ell_A \left([Q_{d-1}M:_Ma_d]/Q_{d-1}M\right),
\end{eqnarray*}
so that 
\begin{eqnarray*}
\e_{Q}^0(M)&=&\e_{QC}^0(C)-\ell_A(A/Q)\\
%&=&\{\ell_A(A/Q)+\ell_A(M/QM)\}-\ell_A \left([Q_{d-1}M:_Ma_d]/Q_{d-1}M\right)-\ell_A(A/Q)\\
&=& \ell_A(M/QM)-\ell_A \left([Q_{d-1}M:_Ma_d]/Q_{d-1}M\right),
\end{eqnarray*}
\begin{eqnarray*}
(-1)^i\e_Q^i(M)=(-1)^i\e_{QC}^i(C) &=& \ell_C(\H_{\n}^0(C/Q_{d-i}C))-\ell_C(\H_{\n}^0(C/Q_{d-i-1}C))\\
&=&\ell_A(\H_{\m}^0(M/Q_{d-i}M))-\ell_A(\H_{\m}^0(M/Q_{d-i-1}M))
\end{eqnarray*}
for $1 \leq i \leq d-1$, and 
$$(-1)^d\e_Q^d(M)=(-1)^{d}\e_{QC}^{d}(C) = \ell_C(\H_{\n}^0(C))=\ell_A(\H_{\m}^0(M)).$$
Hence the result follows.
\end{proof}

We are now in a position to prove Theorem \ref{main1}.

\begin{proof}[Proof of Theorem \ref{main1}]
$(1) \Rightarrow (2)$
Since the last assertion (i) follows from Proposition \ref{d-seq}, we have assertion (b) by Proposition \ref{d-poly}. It is now enough  to show that assertion (a) holds true. We proceed by induction on $d$.
Thanks to Lemma \ref{d=1}, we may assume that $d \geq 2$ and that our assertion holds true for $d-1$.
Choose an element $a \in Q \backslash \m Q$ so that $a$ is  superficial  for $M$ and $M_j$ with respect to $Q$ and $\hdeg_Q(M_j/aM_j) \leq \hdeg_Q(M_j)$ for all $1 \leq j \leq d-1$ (Lemma \ref{superficial1}). We set  $\overline{M}=M/aM$ and $\overline{Q}=Q/(a)$. Then by the same argument as is in the proof of Lemma \ref{superficial1}, we get the inequalities
\begin{eqnarray*}
\hdeg_Q(\overline{M}_j) &\leq& \hdeg_Q([(0):_{M_j}a]) + \hdeg_Q(M_{j+1}/aM_{j+1})
\end{eqnarray*}
and
$$\ell_A([(0):_{M_j}a]) \leq \hdeg_Q(M_j)$$
for all $0 \leq j \leq d-2$.
Hence 
\begin{eqnarray*}
\chi_1(Q;M) = \chi_1(\overline{Q};\overline{M}) &\leq& \hdeg_{Q}(\overline{M}) -\e_Q^0(\overline{M})\\
&=& \sum_{j=0}^{d-2}\binom{d-2}{j}\hdeg_Q(\overline{M}_j)\\
&\leq& \sum_{j=0}^{d-2}\binom{d-2}{j}\{ \hdeg_Q([(0):_{M_j}a]) + \hdeg_Q(M_{j+1}/aM_{j+1})  \}\\
&\leq& \sum_{j=0}^{d-2}\binom{d-2}{j}\{ \hdeg_Q(M_j)+\hdeg_Q(M_{j+1}) \}\\
&=& \sum_{j=0}^{d-1}\binom{d-1}{j}\hdeg_Q(M_j)\\
&=& \hdeg_Q(M) - \e_Q^0(M)=\chi_1(Q;M),
\end{eqnarray*}
because $\chi_1(Q;M)=\chi_1(\overline{Q};\overline{M})$ by Lemma \ref{/a} and $\chi_1(\overline{Q};\overline{M}) \leq \hdeg_{Q}(\overline{M})-\e_Q^0(\overline{M})$ by Proposition \ref{euler}.
Thus $$\chi_1(\overline{Q};\overline{M})=\hdeg_Q(\overline{M})-\e_Q^0(\overline{M}),$$ $$\hdeg_Q(\overline{M}_j)=\hdeg_Q(M_j)+\hdeg_Q(M_{j+1}),$$ and $aM_j=(0)$ for all $0 \leq j \leq d-2$. On the other hand, since $a$ is superficial for $M$ with respect to $Q$, we have $\e_Q^i(M)=\e_Q^i(\overline{M})$ for all $0 \leq i \leq d-2$ and $(-1)^{d-1}\e_Q^{d-1}(M)=(-1)^{d-1}\e_Q^{d-1}(\overline{M})-\ell_A([(0):_M a])$ (\cite[(22.6)]{N}). Therefore the hypothesis of induction on $d$ yields that
\begin{eqnarray*}
(-1)^i\e_Q^i(M)=(-1)^i\e_Q^i(\overline{M})&=&\rmT_Q^i(\overline{M})\\
&=& \sum_{j=1}^{d-1-i}\binom{d-i-2}{j-1}\hdeg_Q(\overline{M}_j)\\
&=& \sum_{j=1}^{d-1-i}\binom{d-i-2}{j-1}\{\hdeg_Q(M_j)+\hdeg_Q(M_{j+1})\}\\
&=& \sum_{j=1}^{d-i}\binom{d-i-1}{j-1}\hdeg_Q(M_j)\\
&=& \rmT_Q^i(M)
\end{eqnarray*}
for $1 \leq i \leq d-2$ and that 
\begin{eqnarray*}
(-1)^{d-1}\e_Q^{d-1}(M)&=&(-1)^{d-1}\e_Q^{d-1}(\overline{M})-\ell_A([(0):_M a])\\
&=& \ell_A(\H_{\m}^0(\overline{M}))-\ell_A(\H_{\m}^0(M))\\
&=& \{\hdeg_Q(M_0)+\hdeg_Q(M_1)\}-\hdeg_Q(M_0)\\
&=&\hdeg_Q(M_1)\\
&=&\rmT_Q^{d-1}(M),
\end{eqnarray*}
because $a\H_{\m}^0(M)=(0)$ and
$\ell_A(\H_{\m}^0(\overline{M})) = \hdeg_Q(\overline{M}_0)=\hdeg_Q(M_0)+\hdeg_Q(M_1).$
Thus, as the equality $ (-1)^d\e_Q^d(M)=\ell_A(\H_{\m}^0(M))$ holds true by Proposition \ref{d-poly}, assertion (a) follows, which proves the implication $(1) \Rightarrow (2)$.

$(2) \Rightarrow (1)$
We have
\begin{eqnarray*}
\ell_A(M/QM) &=& \sum_{i=0}^d(-1)^i\e_Q^i(M)\\
&=& \e_Q^0(M)+\sum_{i=1}^{d-1}(-1)^i\e_Q^i(M)+(-1)^d\e_Q^d(M)\\
&=& \e_Q^0(M)+\sum_{i=1}^{d-1}\rmT_Q^i(M)+\ell_A(\H_{\m}^0(M))
\end{eqnarray*}
by conditions (a) and (b). Hence $\ell_A(\H_{\m}^0(M))=\hdeg_Q(M_0)$ and
\begin{eqnarray*}
\sum_{i=1}^{d-1}\rmT_Q^i(M)&=&\sum_{i=1}^{d-1}\sum_{j=1}^{d-i}\binom{d-i-1}{j-1}\hdeg_Q(M_j)\\
%&=& \sum_{i+j\leq d, 1 \leq i,j}\binom{d-i-1}{j-1}\hdeg_Q(M_j)\\
&=& \sum_{j=1}^{d-1}\left\{\sum_{i=1}^{d-j}\binom{d-i-1}{j-1}\right\}\hdeg_Q(M_j)\\
&=& \sum_{j=1}^{d-1}\left\{\sum_{i=1}^{d-j}\left[\binom{d-i}{j}-\binom{d-i-1}{j}\right]\right\}\hdeg_Q(M_j)\\
&=& \sum_{j=1}^{d-1}\left\{\sum_{i=1}^{d-j}\binom{d-i}{j}-\sum_{i=1}^{d-j-1}\binom{d-i-1}{j}\right\}\hdeg_Q(M_j)\\
&=& \sum_{j=1}^{d-1}\left\{\sum_{i=1}^{d-j}\binom{d-i}{j}-\sum_{i=2}^{d-j}\binom{d-i}{j}\right\}\hdeg_Q(M_j)\\
&=& \sum_{j=1}^{d-1}\binom{d-1}{j}\hdeg_Q(M_j).
\end{eqnarray*}
Thus  
\begin{eqnarray*}
\chi_1(Q;M)=\ell_A(M/QM)-\e_Q^0(M)&=& \sum_{i=1}^{d-1}\rmT_Q^i(M)+\ell_A(\H_{\m}^0(M))\\
&=& \sum_{j=0}^{d-1}\binom{d-1}{j}\hdeg_Q(M_j)\\
&=&\hdeg_Q(M)-\e_Q^0(M),
\end{eqnarray*}
which shows the implication $(2) \Rightarrow (1)$.

We now consider assertion (ii). 
We get $QM \cap \H_{\m}^0(M)=(0)$ by Lemma \ref{W}.
Suppose that $d \geq 3$.
Let $Q=(a_1,a_2,\ldots,a_d)$ and $1 \leq i \leq d$. Since the residue class field $A/\m$ of $A$ is infinite, we may choose the elements $a'_is$ so that $a_i$ is   superficial  for $M$ and $M_j$ with respect to $Q$ and $\hdeg_Q(M_j/a_iM_j) \leq \hdeg_Q(M_j)$ for all $1 \leq j \leq d-1$. Then, thanks to the proof of the implication $(1) \Rightarrow (2)$, $a_iM_j=(0)$ for all $1 \leq j \leq d-2$.
Consequently, by the symmetry of  $a'_is$, $Q\H_{\m}^j(M)=(0)$ for all $1 \leq j \leq d-2$, which proves  assertion (ii) and Theorem \ref{main1}.
\end{proof}

We close this section with the following example of  parameter ideals $Q$ such that $\chi_1(Q;A)=\hdeg_Q(A)-\e_Q^0(A)$ but  $A$ is not a generalized Cohen-Macaulay ring.

\begin{ex}\label{ex1}
Let $\ell \geq 2$ and $m \geq 1$ be integers.
Let $$S=k[[X_i,Y_i,Z_j \ | \ 1 \leq i \leq \ell, 1 \leq j \leq m]]$$
be the formal power series ring with $2\ell+m$ indeterminates over an infinite field $k$. Let $$A=S/(X_1,X_2,\ldots,X_{\ell}) \cap (Y_1,Y_2,\ldots,Y_{\ell}),$$
$$\m=(x_i,y_i,z_j \ | \ 1 \leq i \leq \ell, 1 \leq j \leq m)A, ~~\text{and}$$ 
$$Q=(x_i-y_i \ | \ 1 \leq i \leq \ell)A+(z_j \ | \ 1 \leq j \leq m)A,$$
where $x_i$, $y_i$, and $z_j$ denote the images of $X_i$, $Y_i$, and $Z_j$ in $A$ respectively.
Then  $\m^2=Q\m$, whence  $Q$ is a reduction of $\m$. We furthermore have the following:
\begin{itemize}
\item[$(1)$] $A$ is an unmixed local ring with $\dim A=\ell+m$, $\depth A=m+1$, and $\H_{\m}^{m+1}(A)$ is not  finitely generated.
\item[$(2)$] $\ell_A(A/Q)=\ell+1$, $\e_Q^0(A)=2$, and hence $\chi_1(Q;A)=\ell-1$.
\item[$(3)$] $\hdeg_Q(A)=2+\binom{\ell+m-1}{m+1}$.
\item[$(4)$] Hence $\chi_1(Q;A) = \hdeg_Q(A)-\e_Q^0(A)$, if $\ell=2$.
\end{itemize}
\end{ex}

\begin{proof}
Set $\a_1=(X_i \ | \ 1 \leq i \leq \ell)$ and $\a_2=(Y_i \ | \ 1 \leq i \leq \ell)$ and consider the exact sequence
$$ 0 \to A \to S/\a_1 \times S/\a_2 \to S/[\a_1+\a_2] \to 0$$
of $S$-modules. 
Then because $$S/\a_1 \cong k[[Y_i, Z_j \ | \ 1 \leq i \leq \ell, 1 \leq j \leq m]],$$ 
$$S/\a_2 \cong k[[X_i, Z_j \ | \ 1 \leq i \leq \ell, 1 \leq j \leq m]], \ \ \ \mbox{and}$$ 
$$S/[\a_1+\a_2] \cong k[[Z_j \ | \ 1 \leq j \leq m]],$$ 
we get $\dim A=\ell+m$, $\H_{\m}^{m+1}(A) \cong \H_{\m}^m(S/[\a_1+\a_2])$, and $\H_{\m}^j(A)=0$ for all $j \neq m+1$, $\ell+d$.
Hence 
$$\hdeg_Q(A_{m+1})=\hdeg_Q(S/[\a_1+\a_2])=\e_Q^0(S/[\a_1+\a_2])=\e_{\m}^0(S/[\a_1+\a_2])=1$$ 
and $\hdeg_Q(A_j)=0$ for all $0 \leq j \leq \ell+m-1$ such that  $j \neq m+1$.
Therefore, since $\e_Q^0(A)=\e_{\m}^0(A)=2$, we get
\begin{eqnarray*}
\hdeg_Q(A) &=& \e_Q^0(A)+\sum_{j=0}^{\ell+m-1}\binom{\ell+m-1}{j}\hdeg_Q(A_j)\\
&=& 2+\binom{\ell+m-1}{m+1}, 
\end{eqnarray*}
while $$ \chi_1(Q;A)=\ell_A(A/Q)-\e_Q^0(A)=(\ell+1)-2=\ell-1,$$
because $\ell_A(A/Q)=\ell+1$.
\end{proof}

\section{The first Hilbert coefficients versus the homological torsions}

The purpose of this section is to estimate  the first Hilbert coefficients of parameters in terms of the homological torsions of modules. The following inequality is given by \cite{GhGHOPV2}. We indicate a brief proof for the sake of completeness.

\begin{prop}\label{e1}$(${\cite[Theorem 6.6]{GhGHOPV2}}$)$
Suppose that $d \ge 2$ and let $Q$ be a parameter ideal of $A$.
Then 
$$
\rme_Q^1(M) \ge -\rmT_Q^1(M)
$$
for every finitely generated $A$-module $M$ with $d = \dim_A M$.
\end{prop}

\begin{proof}
We proceed by induction on $d$.
Let $M' = M/\rmH_\m^0(M)$.
Then, since $\e_Q^1(M) = \e_Q^1(M')$ and $\rmT_Q^1(M) = \rmT_Q^1(M')$, to see that $\rme_Q^1(M) \ge -\rmT_Q^1(M)$, we may assume, passing to $M'$, that $\depth_AM >0$. Suppose that $d=2$. Choose $a \in Q \setminus \fkm Q$ so that $a$ is superficial for $M$ and $M_1$ with respect to $Q$ and $\hdeg_Q(M_1 / aM_1) \le \hdeg_Q (M_1)$. Set $\overline{M} = M/aM$. Then since $a$ is $M$--regular, we get the exact sequence
$$ 0 \to \H_{\m}^0(\overline{M}) \to \H_{\m}^1(M) \overset{a}{\to} \H_{\m}^1(M) $$
of local cohomology modules. Taking the Matlis dual, we get an isomorphism $M_1/aM_1 \cong \overline{M}_0$ and hence, because $\e_Q^1(\overline{M})=-\ell_A(\H_{\m}^0(\overline{M}))$ by Lemma \ref{d=1}, we have
\begin{eqnarray*}
\e_Q^1(M) = \e_Q^1(\overline{M}) &=& - \ell_A(\H_{\m}^0(\overline{M}))\\
&=& -\hdeg_Q(\overline{M}_0)\\
&=& - \hdeg_Q(M_1/aM_1)\\
&\geq& - \hdeg_Q(M_1)\\
&=&-\rmT_Q^1(M).
\end{eqnarray*}
Suppose that $d \geq 3$ and that our assertion holds true for $d-1$. Choose $a \in Q \setminus \fkm Q$ so that $a$ is superficial for $M$ with respect to $Q$ and $\rmT_Q^1(\overline{M}) \le \rmT_Q^1(M)$ (Lemma \ref{superficial2}). Then the hypothesis of induction on $d$ shows
$$\rme_Q^1(M)=\rme_Q^1(\overline{M}) \geq -\rmT_Q^1(\overline{M}) \geq - \rmT_Q^1(M),$$ 
as wanted.
\end{proof}

The first Hilbert coefficients $\e_Q^1(M)$ of parameter ideals are bounded below by the homological torsion $-\rmT_Q^1(M)$. It is now natural to ask what happens on the parameters $Q$ of $M$, once the equality $\e_Q^1(M)=-\rmT_Q^1(M)$ is attained. The main result of this section answers the question and is stated as follows. Recall that  a finitely generated $A$-module $M$ is said to be unmixed, if $\dim A/\p=\dim_AM$ for all $\p \in \Ass_{A}M$.

\begin{thm}\label{main2}
Let $M$ be a finitely generated $A$-module with $d = \dim_AM \geq 2$ and suppose that $M$ is unmixed. 
Let $Q$ be a parameter ideal of $A$. 
Then the following conditions are equivalent:
\begin{itemize}
\item[$(1)$] $\chi_1(Q;M)=\hdeg_Q(M)-\e_Q^0(M)$.
\item[$(2)$] $\e_Q^1(M)=-\rmT_Q^1(M)$.
\end{itemize}
When this is the case, we have the following$\mathrm{:}$
\begin{itemize}
\item[$(\mathrm{i})$] $(-1)^i\e_Q^i(M)=\rmT_Q^i(M)$ for $2 \leq i\leq d-1$ and $\e_Q^d(M)=0$.
\item[$(\mathrm{ii})$] There exist elements $a_1,a_2,\ldots,a_d \in A$ such that $Q=(a_1,a_2,\ldots,a_d)$ and $a_1,a_2,\ldots,a_d$ forms a $d$-sequence on $M$.
\item[$(\mathrm{iii})$] $Q\H_{\m}^i(M)=(0)$ for all $1 \leq i \leq d-2$.
\end{itemize}
\end{thm}

To prove Theorem \ref{main2}, we need the following:

\begin{prop}\label{unmixed}$(${\cite[Theorem 2.5]{GhGHOPV2}}$)$
Let $M$ be a finitely generated $A$-module with $d = \dim_AM$.
Suppose that $M$ is unmixed.
Then there exist a surjective homomorphism $B \to A$ of rings such that $B$ is a Gorenstein complete local ring with $\dim B=\dim A$ and an exact sequence
$$0 \to M \to F \to X \to 0$$
of $B$-modules with $F$ finitely generated and free.
\end{prop}

As a direct consequence we get the following.

\begin{cor}\label{GNa}$(${\cite[Lemma 3.1]{GNa}}$)$
Let $M$ be a finitely generated $A$-module with $d = \dim_AM\geq 2$. If $M$ is unmixed, then $\H_{\m}^1(M)$ is finitely generated.
\end{cor}

Let $\Assh_AM=\{\p \in \Ass_AM \ | \ \dim A/\p=\dim_A M \}$ and let $(0)=\bigcap_{\p \in \Ass_AM}\operatorname{I}(\p)$ be a primary decomposition of $(0)$ in $M$, where for each $\p \in \Ass_AM$, $\operatorname{I}(\p)$ denotes a $\p$-primary submodule of $M$.
Set
$$\rmU_M(0)=\bigcap_{\p \in \Assh_AM}\operatorname{I}(\p)$$
and call it the unmixed component of $(0)$ in $M$.

We are now in a position to prove Theorem \ref{main2}.

\begin{proof}[Proof of Theorem \ref{main2}]
Thanks to Theorem \ref{main1}, we have only to show the implication $(2) \Rightarrow (1)$. We proceed by induction on $d$. Suppose that $d=2$. Then by Corollary \ref{GNa}, $M$ is a generalized Cohen-Macaulay $A$-module and
$$\e_Q^1(M)=-\rmT_Q^1(M)=-\hdeg_Q(M_1)=-\ell_A(\H_{\m}^1(M)).$$ Therefore $Q$ is a standard parameter ideal for $M$ by \cite[Theorem 2.1]{GO1} and the required equality $\chi_1(Q;M)=\hdeg_Q(M)-\e_Q^0(M)$ follows. Assume that $d \geq 3$ and that our assertion holds true for $d-1$. By Proposition \ref{unmixed} we may assume that $A$ is a Gorenstein local ring and that there exists an exact sequence
$$0 \to M \overset{\varphi}{\to} F \to X \to 0 \ \ \ \ (\sharp_4)$$
of $A$-modules with $F$ a finitely generated free $A$-module and $X =\operatorname{Coker} \varphi$.
Since the residue class field $A/\m$ of $A$ is infinite, we may now choose an element $a \in Q \backslash \m Q$ so that $a$ is superficial for $M$, $F$, $X$, and $M_j$ with respect to $Q$ and  $\hdeg_Q(M_j/aM_j) \leq \hdeg_Q(M_j)$ for all $1 \leq j \leq d-1$.
Set $\overline{M}=M/aM$ and $\overline{Q}=Q/(a)$.
Then, by the same argument as is in the proof of Lemma \ref{superficial1}, we have
$$ \ell_A([(0):_{M_j} a]) \leq \hdeg_Q(M_j)  \ \ \ \text{and}$$
$$ \hdeg_Q(\overline{M_j}) \leq \ell_A([(0):_{M_j} a]) +\hdeg_Q(M_{j+1}/aM_{j+1}) $$
for all $1 \leq j \leq d-2$.

We consider the exact sequence
$$ 0 \to [(0):_Xa] \to \overline{M} \overset{\overline{\varphi}}{\to} F/aF \to X/aX \to 0 $$
of $A$-modules obtained by  exact sequence $(\sharp_4)$,  where $\overline{\varphi}=A/(a) \otimes \varphi$.
Set $L=\Im \overline{\varphi}$. Then since $L$ is unmixed with $\dim_AL=d-1$ and $\ell_A([(0):_Xa]) < \infty$, we get $$ [(0):_Xa] \cong \H_{\m}^0(\overline{M})=\rmU_{\overline{M}}(0), $$
where $U=\rmU_{\overline{M}}(0)$ denotes the unmixed component of $(0)$ in $\overline{M}$. Consequently, because $a$ is superficial for $M$ with respect to $Q$ and $L \cong \overline{M}/U$ with $\ell_A(U) < \infty$, we see $\e_Q^i(M)=\e_Q^i(\overline{M})=\e_Q^i(L)$ for $i=0, 1$ and $\H_{\m}^j(L) \cong \H_{\m}^j(\overline{M})$ for all $j \geq 1$.
Hence $\hdeg_Q(L_j)=\hdeg_Q(\overline{M}_j)$ for all $1 \leq j \leq d-2$. Therefore 
\begin{eqnarray*}
\e_Q^1(M)=\e_Q^1(\overline{M})=\e_Q^1(L) &\geq& -\rmT_Q^1(L)\\
&=& -\sum_{j=1}^{d-2}\binom{d-3}{j-1}\hdeg_Q(L_j)\\
&=& -\sum_{j=1}^{d-2}\binom{d-3}{j-1}\hdeg_Q(\overline{M}_j)\\
&\geq& -\sum_{j=1}^{d-2}\binom{d-3}{j-1}\{ \ell_A([(0):_{M_j} a]) +\hdeg_Q(M_{j+1}/aM_{j+1}) \}\\
&\geq& -\sum_{j=1}^{d-2}\binom{d-3}{j-1}\{\hdeg_Q(M_j)+\hdeg_Q(M_{j+1})\}\\
&=& -\sum_{j=1}^{d-1}\binom{d-2}{j-1}\hdeg_Q(M_j)\\
&=& -\rmT_Q^1(M)=\e_Q^1(M),
\end{eqnarray*}
because $\e_Q^1(L) \geq - \rmT_{Q}^1(L)$ by Proposition \ref{e1}. Thus $\e_Q^1(L)=-\rmT_Q^1(L)$, $\hdeg_Q(\overline{M}_j)=\hdeg_Q(M_j)+\hdeg_Q(M_{j+1})$, and $aM_j=(0)$ for all $1 \leq j \leq d-2$, so that the hypothesis of induction on $d$ yields 
\begin{eqnarray*}
\chi_1(\overline{Q}:L)&=&\hdeg_{Q}(L)-\e_{Q}^0(L)\\
&=& \{\hdeg_{Q}(\overline{M})-\ell_A(U)\}-\e_{Q}^0(\overline{M})\\
&=& \sum_{j=0}^{d-2}\binom{d-2}{j}\hdeg_Q(\overline{M}_j) -\ell_A(\H_{\m}^0(\overline{M}))\\
&=& \sum_{j=0}^{d-2}\binom{d-2}{j}\{\hdeg_Q(M_j)+\hdeg_Q(M_{j+1})\} -\ell_A(\H_{\m}^0(\overline{M}))\\
&=& \sum_{j=0}^{d-1}\binom{d-1}{j}\hdeg_Q(M_j) -\ell_A(\H_{\m}^0(\overline{M}))\\
&=& \hdeg_Q(M)-\e_Q^0(M)-\ell_A(\H_{\m}^0(\overline{M})) \ \ \ \ \ \ \ \ \ \ \ \ \ (\dagger_1),
\end{eqnarray*}
because $\hdeg_Q(\overline{M})=\hdeg_Q(L)+\ell_A(U)$ by Lemma \ref{xyz} (2) and $\ell_A(U)=\ell_A(\H_{\m}^0(\overline{M}))$. We also have 
$$\ell_A(\overline{M}/Q\overline{M})=\ell_A(L/QL)+\ell_A(U/Q\overline{M} \cap U)\ \ \ \ \ \ \ \ \ \ \ \ \ (\dagger_2)$$
by the exact sequence
$$ 0 \to U/Q\overline{M} \cap U \to \overline{M}/Q\overline{M} \to L/QL \to 0 $$
obtained from the exact sequence $$ 0 \to U \to \overline{M} \to L \to 0.$$ Therefore by $(\dagger_1)$ and $(\dagger_2)$ we get 
\begin{eqnarray*}
\chi_1(Q;M)&=&\chi_1(\overline{Q};\overline{M})\\
&=&\ell_A(\overline{M}/Q \overline{M})-\e_{Q}^0(\overline{M})\\
&=& \{\ell_A(L/QL)+\ell_A(U/Q\overline{M} \cap U)\}-\e_{Q}^0(L)\\
&=& \chi_1(\overline{Q};L)+\ell_A(U/Q\overline{M} \cap U)\\
&=& \{ \hdeg_Q(M)-\e_Q^0(M)-\ell_A(\H_{\m}^0(\overline{M})) \}+\{\ell_A(U)-\ell_A(Q\overline{M} \cap U)\}\\
&=& \hdeg_Q(M)-\e_Q^0(M)-\ell_A(Q\overline{M} \cap U),
\end{eqnarray*}
because $\chi_1(Q;M)=\chi_1(\overline{Q};\overline{M})$ by Lemma \ref{/a} and $\ell_A(U)=\ell_A(\H_{\m}^0(\overline{M}))$. Thus, to prove $\chi_1(Q;M)=\hdeg_Q(M)-\e_Q^0(M)$, it is enough to show that $Q\overline{M} \cap U=(0)$.

Let us choose an element $b \in Q \backslash \m Q$ so that $b$ is  superficial for $M$, $F$, $X$, and $M_j$ with respect to $Q$, $\hdeg_Q(M_j/bM_j) \leq \hdeg_Q(M_i)$ for all $1 \leq j \leq d-1$, and $a$, $b$ forms a part of a minimal system of generators of $Q$.
We set $\overline{M}'=M/bM$, $\overline{A}'=A/(b)$, and $\overline{Q}'=Q/(b)$. Then, tensoring $(\sharp_4)$ by $\overline{A}'=A/(b)$, we get the exact sequence
$$ 0 \to [(0):_Xb] \to \overline{M}' \overset{\overline{\varphi}'}{\to} F/bF \to X/bX \to 0, $$ 
where $\overline{\varphi}'=A/(b) \otimes \varphi$. Set $L'=\Im \overline{\varphi}'$. Then because $L'$ is unmixed with $\dim_AL'=d-1$ and $\ell_A([(0):_X b]) < \infty$, we have
$$ [(0):_X b] \cong \H_{\m}^0(\overline{M}')=\rmU_{\overline{M}'}(0),$$
where $U'=\rmU_{\overline{M}'}(0)$ is the unmixed component of $(0)$ in $\overline{M}'$. Consequently by the same argument as above, $\e_Q^1(L')=-\rmT_Q^1(L')$ and $bM_i=(0)$ for all $1 \leq i \leq d-2$, so that thanks to the hypothesis of induction on $d$, we get $\chi_1(\overline{Q}';L')=\hdeg_{\overline{Q}'}(L')-\e_{\overline{Q}'}^0(L')$.

We now choose the element $a \in Q \setminus \m Q$ to be superficial also for $L'$ with respect to $\overline{Q}'$ and $\hdeg_{\overline{Q}'}(L'/aL') \leq \hdeg_{\overline{Q}'}(L')$.
Then by Proposition \ref{d-seq} there exist elements $a_3,a_4,\ldots,a_d \in A$ such that $\overline{Q}'=(a,a_3,\ldots,a_d)\overline{A}'$ and $a,a_3,a_4,\ldots,a_d$ forms a $d$-sequence on $L'$, because $\chi_1(\overline{Q}';L')=\hdeg_{\overline{Q}'}(L')-\e_{\overline{Q}'}^0(L')$. Take $\alpha \in Q\overline{M} \cap U$ and write $\alpha=\overline{x}$ with $x \in QM \cap \rmU(aM)$, where $$\rmU(N)=\bigcup_{\ell>0}[N:_M \m^{\ell}]$$ for each submodule $N$ of $M$ and $\overline{x}$ denotes the image of $x$ in $\overline{M}$. Let us  consider the composite of the canonical maps
$$ \rho:M \to \overline{M}' \to L' \to L'/aL'.$$ Then $\m^{\ell}x \subseteq aM$ for all $\ell \gg 0$ and $x \in QM$.  Therefore 
$$ \rho(x) \in \H_{\m}^0(L'/aL') \cap (a,a_3,\ldots,a_d)(L'/aL')=(0),$$
because $a,a_3,a_4,\ldots,a_d$ forms a $d$-sequence on $L'$. Consequently, $x \in aM+\rmU(bM) \cap \rmU(aM)$.
Let us write $x=y+z$ with $y \in aM$ and $z \in \rmU(bM) \cap \rmU(aM)$. Then because $a$ and $b$ are $M$-regular, we have the embeddings $$ \rmU(aM)/aM =U=\H_{\m}^0(\overline{M}) \hookrightarrow \H_{\m}^1(M),$$
$$ \rmU(bM)/bM =U'=\H_{\m}^0(\overline{M}') \hookrightarrow \H_{\m}^1(M),$$ so that $b\rmU(aM) \subseteq aM$ and $a\rmU(bM) \subseteq bM$, since $a\H_{\m}^1(M)=b\H_{\m}^1(M)=(0)$.
Therefore  $ az \in bM$ and $bz \in aM$.
We now write 
$$ az = bv \ \ \ \mbox{and} \ \ \ bz =aw$$
with $v,w \in M$. 
Then $v \in [a^2M:_Mb^2]$, since $abz=b^2v=a^2w$.

\begin{claim}\label{claim}
$[a^2M:_Mb^2] \subseteq [a^2M:_M b]$.
\end{claim}

\noindent
{\it{Proof of Claim \ref{claim}.}}
Tensoring exact sequence $(\sharp_4)$ by $A/(a^2)$, we get the exact sequence
$$ 0 \to [(0):_X a^2] \to M/a^2M \overset{\widetilde{\varphi}}{\to} F/a^2F \to X/a^2X \to 0,$$
where $\widetilde{\varphi}=A/(a^2) \otimes \varphi$. Since $\ell_A([(0):_Xa]) < \infty$, $[(0):_X a]_{\p}=(0)$ for all $\p \in \Spec A \backslash \{\m\}$.
Hence $a$ is $X_{\p}$-regular, so that $\ell_A([(0):_Ma^2])< \infty$.
Therefore because $\depth_AF/a^2F >0$, we get an isomorphism $$[(0):_X a^2] \cong \H_{\m}^0(M/a^2M).$$
Take $\xi \in [a^2M:_Mb^2]$ and let $\overline{\xi}$ denotes the image of $\xi$ in $M/a^2M$.
Then $b^2\widetilde{\varphi}(\overline{\xi})=\widetilde{\varphi}(b^2\overline{\xi})=0$ in $F/a^2F$, whence $\widetilde{\varphi}(\overline{\xi})=0$, because $a^2,b^2$ forms an $F$-regular sequence.
Therefore $$\overline{\xi} \in \ker \widetilde{\varphi} \cong \H_{\m}^0(M/a^2M) \hookrightarrow \H_{\m}^1(M)$$
and hence $b \overline{\xi}=0$ in $M/a^2M$, because $b \H_{\m}^1(M)=(0)$.
Thus $b\xi \in a^2M$, so that $\xi \in [a^2M:_M b]$.
Consequently $[a^2M:_Mb^2] \subseteq [a^2M:_M b]$, which proves Claim \ref{claim}.\\

We have $v \in [a^2M:_M b]$ by Claim \ref{claim}.
Hence $bv \in a^2M$. We write $bv=a^2v'$ with $v' \in M$.
Then $z=av' \in aM$, since $abz=b^2v=a^2bv'$ and $\depth_AM >0$. Therefore $x=y+z \in aM$, so that $\alpha=\overline{x}=0$ in $\overline{M}$. Thus $Q\overline{M} \cap U=(0)$. Hence the required equality $\chi_1(Q;M)=\hdeg_Q(M)-\e_Q^0(M)$ follows, which completes the proof of Theorem \ref{main2} as well as the proof of the implication $(2) \Rightarrow (1)$.
\end{proof}

The following example shows that the implication $(2) \Rightarrow (1)$ does not hold true in general, unless $M$ is unmixed.

\begin{ex}\label{ex2}
Let $S$ be a complete regular local ring with maximal ideal $\n$, $\dim S = 3$, and infinite residue class field. 
Let $\fkn = (X,Y,Z)$ and $\ell \geq 1$ be integers.
We set $$A=S/(X) \cap (Y^{\ell}, Z).$$ Let $\m=(x,y,z)A$ be the maximal ideal of $A$ and $Q=(x-y,x-z)A$, where $x$, $y$, and $z$ denote the images of $X$, $Y$, and $Z$ in $A$, respectively. Then, since $\m^{\ell+1}=Q\m^{\ell}$, $Q$ is a reduction of $\m$. We furthermore have the following:

\begin{itemize}
\item[$(1)$] $A$ is mixed with $\dim A=2$ and $\depth A=1$,
\item[$(2)$] $\e_Q^0(A)=1$, $\e_Q^1(A)=-\ell$, and $\e_Q^2(A)=-\binom{\ell}{2}$,
\item[$(3)$] $\chi_1(Q;A)=1$, $\hdeg_Q(A)=\ell+1$, and $\rmT_Q^1(A)=\ell$.
\item[$(4)$] Hence $\e_Q^1(A)=-\rmT_Q^1(A)$, and if $\ell=1$, $\chi_1(Q;A)= \hdeg_Q(A)-\e_Q^0(A)$, but if $\ell \geq 2$, $\chi_1(Q;A)< \hdeg_Q(A)-\e_Q^0(A)$.
\end{itemize}
\end{ex}

\begin{proof}
Consider the canonical exact sequence
$$ 0 \to xA \to A \to A/xA \to 0.$$
Set $\a=(y^{\ell},z)A$. Then $U=xA ~(\cong A/\a$) is the unmixed component of $(0)$ in $A$.
Set $B=A/xA$.
Then since $B$ is a regular local ring with $\dim B=2$ and $QB=\m B$, we have
\begin{eqnarray*}
\ell_A(A/Q^{n+1}) &=& \ell_A(B/\m^{n+1}B)+\ell_A(U/Q^{n+1}U)\\
&=& \binom{n+2}{2}+\left[ \e_Q^0(U)\binom{n+1}{1}-\e_Q^1(U) \right]
\end{eqnarray*}
for all $n \gg 0$.

Because the Hilbert series $\rmH(\gr_\m(A/\a), \lambda)$ of the associated graded ring $\gr_\m(A/\a)$ is given by
$$
\rmH(\gr_\m(A/\a) , \lambda) = \frac{1+\lambda + \cdots +\lambda^{\ell-1}}{1-\lambda}
$$
and $Q{\cdot}(A/\a)=\m{\cdot}(A/\a)$, we have $\e_Q^0(U)=\e_{\m}^0(A/\a) = \ell$ and $\e_Q^1(U)=\e_\m^1(A/\a) = \binom{\ell}{2}$.
Therefore
$$
(-1)^i\e_Q^i(A) = \left\{
\begin{array}{lc}
1 & \mbox{if $i = 0$},\\
\vspace{1mm}\\
\rme_Q^0(U) = \ell & \mbox{if $i = 1$},\\
\vspace{1mm}\\
-\rme_Q^1(U) = -\binom{\ell}{2} & \mbox{if $i = 2$}.
\end{array}
\right.
$$ On the other hand, since  $A/\a$ is a Gorenstein ring and
$$
\rmH_\fkm^1(A) \cong \rmH_\m^1(A/\a),
$$
we get 
$$\hdeg_Q(A_1) = \hdeg_Q(A/\a)=\e_Q^0(A/\a)=\e_{\m}^0(A/\a) = \ell.$$
Therefore 
$$ \chi_1(Q;A) =\ell_A(A/Q)-\e_Q^0(A)=2-1=1, $$
$$ \hdeg_Q(A)=\e_Q^0(A)+\hdeg_Q(A_1)=1+\ell, \ \ \  \mbox{and}$$
$$ \rmT_Q^1(A)=\hdeg_Q(A_1)=\ell,$$
since $\ell_A(A/Q)=2$.%Thus $\e_Q^1(A)=-\rmT_Q^1(A)$, but $\chi_1(Q;A)< \hdeg_Q(A)-\e_Q^0(A)$ if $\ell \geq 2$.
\end{proof}

%%%%%%%%%%%%%%%%%%%%%%%%%%%%%%%%%%%%%%%%%%%%%%%%%%%%%%%%%%%%
%\addcontentsline{toc}{section}{references}

\end{document}